\documentclass[12pt]{amsart}%
\usepackage{amsfonts}
\usepackage{amsmath}
\usepackage{soul}
\setstcolor{purple}
\usepackage{color}
\usepackage{tikz,pgf}
\usepackage{amssymb}
\usepackage{graphicx}%
\usepackage[all]{xy}\SelectTips{cm}{}
\definecolor{lightblue}{rgb}{.30,.95,.2}
\sethlcolor{lightblue}
\newcommand{\mb}[1]{\mathbf{ #1}}

\newcommand{\alga}{\mathbf A}
\newcommand{\algb}{\mathbf B}

\newcommand{\algd}{\mathbf D}

\providecommand{\U}[1]{\protect\rule{.1in}{.1in}}
\newtheorem{theorem}{Theorem}
\theoremstyle{plain}

\newtheorem{corollary}{Corollary}

\newtheorem{definition}{Definition}
\newtheorem{example}{Example}

\newtheorem{lemma}{Lemma}

\newtheorem{proposition}{Proposition}
\newtheorem{remark}{Remark}

\numberwithin{equation}{section}
\begin{document}
\title{Orthogonal relational systems}
\author[Bonzio S.]{S. Bonzio}
\address{Stefano Bonzio, University of Cagliari\\
Italy}
\email{stefano.bonzio@gmail.com}

\author[Chajda I.]{I. Chajda}
\address{Ivan Chajda, Palack\'{y} University, Olomouc\\
Czech Republic}
\email{ivan.chajda@upol.cz}

\author[Ledda A.]{A. Ledda}
\address{Antonio Ledda, University of Cagliari\\
Italy}
\email{antonio.ledda@unica.it}
%
%
%\author[Paoli F.]{F. Paoli}
%\address{Francesco Paoli, University of Cagliari\\
%Italy}
%\email{paoli@unica.it}

\begin{abstract}
In this paper we discuss the concept of relational system with involution. This system is called orthogonal if, for every pair of non-zero orthogonal elements, there exists a supremal element in their upper cone and the upper cone of orthogonal elements $x,\,x'$ is a singleton (i.e. $x,\,x'$ are complements each other). To every orthogonal relational system can be assigned a groupoid with involution. The conditions under which a groupoid is assigned to an orthogonal relational systems are investigated. 
We will see that many properties of the relational system can be captured by the associated groupoid. Moreover, these structures enjoy several desirable algebraic features such as, e.g., a direct decomposition representation and the strong amalgamation property.
\end{abstract}
%\thanks{{\bf Corresponding author}: {\bf Stefano Bonzio}, {\ttfamily stefano.bonzio@gmail.com}}
\maketitle

\section{Introduction}

It is superfluous to recall how important binary relational systems are for the whole of mathematics. The study of binary relations traces back to the work of J. Riguet \cite{Rig48}, and a first attempt to provide an algebraic theory of relational systems is due to Mal'cev \cite{Malcev}. A general investigation of quotients and homomorphisms of relational systems can be found in \cite{ChLa10}, where seminal notions from \cite{Ch04} are developed.
A  leading motivation for our discussion stems from the theory of \emph{semilattices}. In fact, semilattices can be equivalently presented as ordered sets as well as groupoids. This approach was widen to ordered sets whose ordering is directed. In this case the resulting groupoid needs not be, in general, a semilattice, but a \emph{directoid} (for details see \cite{Chbook}). We will see that many features of a relational system $\mathbf A=\langle A, R\rangle$ can be captured by means of the associated groupoid. Reflexivity, symmetry, transitivity or antisymmetry of $R$ can be equationally or quasi-equationally characterized in the groupoid \cite{ChLa13, ChLa14}. 

The concept of orthogonal poset was first considered in \cite{Ch14}, where an algebraic characterization of the system through the associated groupoid with involution is presented. In \cite{CGKGLP} this method was  generalized to cover the case of ordered sets with antitone involution. These ideas motivated us to extend the approach to general algebraic systems with involution and distinguished elements. In what follows, we develop this theory.

The paper is structured as follows: in \S\ 2 we present the notions of orthogonal relational system and orthogonal groupoids and show how the two concepts are mutually related. In \S\ 3 we present a decomposition theorem for a variety of orthogonal groupoids. Finally in \S\ 4 we show that the class of orthogonal groupoids enjoys the strong amalgamation property. 

\section{Relational systems with involution}
By a \emph{relational system} is meant a pair $\mathbf A=\langle A,R\rangle$, where $A$ is a non-empty set and $R$ is a binary relation on $A$, i.e. $R\subseteq A^{2}$. 
If $a,b\in A$, the {\it upper cone of $a,b$} is the set
\[
U_{R}(a,b)=\{c\in A:(a,c)\in R\text{ and }(b,c)\in R\}.
\]
In case $a=b$ we write $U_{R}(a)$ for $U_{R}(a,a)$.

A \emph{relational system with involution} is a triple $\mathbf A=\langle A,R,'\rangle$ such that $\langle A,R\rangle$ is a relational system and $':A\to A$ is a map such that, for all $a,b\in A$, $(a')'=a$, and if $(a,b)\in R$ then $(b',a')\in R$. 
For brevity sake, we will write $a''$ for $(a')'$. \\
A \emph{relational system with 1 and involution} is a quadruple $ \mathbf A=\langle A,R,', 1 \rangle $, such that the structure $ \langle A,R,'\rangle$ is a relational system with involution and $1$ is a constant in $A$ such that $ (x,1)\in R $ for each $ x\in A $. 

As customary, we indicate $ 1'$ by $ 0 $. Since $ (x,1)\in R $, then it follows that $ (0,x)\in R $ for all $ x\in A $.  One can easily see that, for any $ a,b\in A $, $ U_{R}(a,b)\neq\emptyset $, as $ 1\in U_{R}(a,b) $. %Throughout the paper $ \mathbf A=\langle A,R,', 1 \rangle $ refers to a relational system with 1 and involution. \\

Let $ \mathbf A=\langle A,R,', 1 \rangle $ and $ a,b\in A $. The elements $ a,b $ are \emph{orthogonal} (in symbols $ a\perp b $) when $ (a,b')\in R $ (or, equivalently, $ (b,a')\in R $). We say that an element $ w\in U_{R}(a,b) $ is a \emph{supremal element} for $ a,b $ if for each $ z\in U_{R}(a,b) $, with $ z\neq w $, then $ (w,z)\in R $. 
Obviously, if $ R $ is an order relation on $ A $, then the supremal element for $ a,b \in A$ coincides with $\sup(a,b) $. 

The following notion will be central in our discussion: 
\begin{definition}\label{sistema ortogonale}
A relational system $ \mathbf A=\langle A,R,', 1 \rangle $ is \emph{orthogonal} if: 
\begin{itemize}
\item[(a)] $ U_{R}(x,x')=\{1\} $ for each $ x\in A $;
\item[(b)] for all $x,y$, if $ x\perp y $ and $ x\neq 0\neq y $ then a supremal element for $ x,y $ exists.
\end{itemize}
\end{definition}
Let us recall a useful notion from \cite{ChLa13} and \cite{ChLa14}. 
\begin{definition}\label{gruppoide indotto}
Let $ \alga=\langle A,R\rangle  $ be a relational system. A binary operation $ + $ on $ A $ can be associated to $R$ as follows: 
\begin{itemize}
\item[(i)] if $ (x,y)\in R $ then $ x+y = y $;
\item[(ii)] if $ (x,y)\not\in R $ and $ (y,x)\in R $ then $ x+y=x $;
\item[(iii)] if $ (x,y)\not\in R $ and $ (y,x)\not\in R $ then $ x+y= y +x =z$, where $ z $ is an arbitrarily chosen element in $ U_{R}(x,y) $.
\end{itemize}
We call the structure $\mathbf{G}(A) = \langle A, + \rangle $ \emph{a groupoid induced} by the relational system $\alga=\langle A,R\rangle$.\footnote{The notation $\mathbf{G}(A)$ involves a mild notational abuse, since, as we shall see the groupoid induced by $\langle A,R\rangle$ is not necessarily unique.}
\end{definition}
Let us remark that, in general, for a relational system $ \alga=\langle A,R\rangle$, an induced groupoid $\mathbf{G}(A)$ is not univocally determined. This happens whenever there are elements $a,b$ in $A$ s.t. $(a,b), (b,a)\notin R$ and $U_R(a,b)$ contains more than one element. In this case indeed, $a+b$ will be arbitrarily chosen in $U_R(a,b)$.

Conversely, if an induced groupoid $\mathbf{G}(A)$ is given, then a relation $ R $ on $ A $ is uniquely determined by the binary operation $ + $ as follows: 
\[
(x,y)\in R\text{ if and only if }x+y= y,
\]

see \cite{ChLa13} and \cite{ChLa14} for details. 

In other words, any induced groupoid $ \mathbf{G}(A) $ stores all the information relative to the relational system  $ \alga=\langle A,R\rangle  $. Furthermore, whenever $ R $ is reflexive, the following obtains:

%-------------------------------------------------
\begin{lemma}\label{x+y in U}
Let $\alga=\langle A,R\rangle  $ be a relational system and $R$ be a reflexive relation. Then $ x+y\in U_{R}(x,y) $ for all $ x,y\in A $.

\proof
If $ (x,y)\in R $ then, by Definition \ref{gruppoide indotto}-(i), $x+y=y$. Therefore, $(x, x+y)\in R$. Moreover, since $R$ is reflexive $(y,y)=(y, x+y)\in R$.
If $ (x,y)\not\in R $ and $ (y,x)\in R $ then, by Definition \ref{gruppoide indotto}-(ii), $x + y=x$. Therefore, $ (y,x)=(y,x+y)\in R $. Moreover, by reflexivity, $(x,x)=(x, x+y)\in R$.
Finally, if $ (x,y)\not\in R $ and $ (y,x)\not\in R $, the claim follows from Definition \ref{gruppoide indotto}-(iii).  
\endproof
\end{lemma}
%-------------------------------------------------
Given a groupoid $ \mathbf{G}=\langle G, +\rangle $ it is possible to define a binary relation $ R_{G} $ on $ G $ as follows, for any $a,b\in G$:  
\[
(a,b)\in R_{G} \text{ if and only if }a+b=b.
\]
%-------------------------------------------------
We call the relational system $ \mathbf{A}(G)=\langle G, R_{G}\rangle $ the \emph{induced relational system} by $\mathbf G $ and $ R_{G} $ the relation \emph{induced} by the groupoid $\mathbf G$.
For simplicity sake, whenever no danger of confusion is impending we drop subscripts from our notation.\\
%-------------------------------------------------
Since Definition \ref{gruppoide indotto}, it is possible to associate an algebra (in particular a groupoid) to any relational system. However, since our aim is to obtain an algebra out of an \emph{orthogonal} relational system, we need to integrate this definition with a further condition, that takes into account orthogonality. 

\begin{definition}\label{def: gruppoide indotto da un ORS}
Let $ \mathbf A=\langle A,R,', 1 \rangle $ be an orthogonal relational system. Then a we associate to $R$ a binary operation $ + $ on $ A $ satisfying conditions (i), (ii), (iii) of Definition \ref{gruppoide indotto} and the following further condition: 
\begin{itemize}
\item[(iv)] if $ x\perp y $ with $ x\neq 0\neq y $, then $ x + y = y + x = w $,
\end{itemize}
where $ w $ is a supremal element in $ U_R(x,y) $.We call the structure $\mathbf{G}(A) = \langle A, +, ^{'}, 1 \rangle $ \emph{a groupoid induced} by the orthogonal relational system $\alga=\langle A,R, ', 1\rangle$. 
\end{definition}
Let us remark that the existence of a supremal element for a pair of orthogonal elements is guaranteed by Definition \ref{sistema ortogonale}.

We can now propose an algebraic counterpart of the notion of orthogonal relational system.
%-------------------------------------------------
\begin{definition}\label{gruppoide ortogonale}
An \emph{orthogonal groupoid}, for short \emph{orthogroupoid}, is an algebra $ \mathbf{D}=\langle D, +, ', 1\rangle $ of type $ (2,1,0) $ such that $ \langle D,+\rangle $ is a groupoid and the following conditions hold:
\begin{itemize}
\item[(a)] $ x''\approx x $;
\item[(b)] $0+x\approx  x $ and $x+1\approx 1 $, where $ 0=1' $;
\item[(c)] $ x+ x'\approx 1 $;
\item[(d)] if $ x+z\approx  z $ and $ x'+ z\approx  z $ then $ z\approx 1 $;
\item[(e)] $ (((z+y)'+(z+x))' + (z+y)') +z' \approx z' $;
\item[(f)] $ x+(x+y) \approx x+y $ and $ y+(x+y)\approx x+y $.
\end{itemize}
\end{definition} 
%-------------------------------------------------
%-------------------------------------------------
Some basic properties of orthogroupoids are subsumed in the following lemmas.
%-------------------------------------------------
\begin{lemma}\label{aritmetica1}
Let $ \mathbf{D}=\langle D, +, ', 1\rangle $ be a groupoid in the type $\langle 2,1,0\rangle$ satisfying conditions (a), (b), (c) and (e) of Definition \ref{gruppoide ortogonale} and $ R $ its induced relation. Then
\begin{itemize}
\item[(i)] $ 0' = 1$.
\item[(ii)] $ (x'+y)' + x = x $.
\item[(iii)] $ (0,x)\in R $ and $ (x,1)\in R $ for any $ x\in D $.
\item[(iv)] If $ (x,y)\in R $ then $ (y',x') \in R $.
\end{itemize}
\proof
(i) $ 0=1' $, thus $ 0'=1''=1 $. \\
(ii) Replacing $ x $ by $ y $ and $ z $ by $ x' $ in Definition \ref{gruppoide ortogonale}-(e), we get $ (((x'+y)'+(x'+y))'+(x'+y)')+x=x $. By (c) and (a) $ (x'+y)'+(x'+y)=1  $, thus $ ((x'+y)'+(x'+y))'=0  $. Then by (b) $(0+ (x'+y)') + x=(x'+y)' + x = x $. \\
(iii) Straightforward from the definition of induced relation.\\
(iv) Let $ (x,y)\in R $. Then, by definition of $R$, $ x+y= y $. By Definition \ref{gruppoide ortogonale}-(a) and item (ii) $y'+x'= (x+y)'+x'=x' $. Therefore $ (y',x')\in R $.
\endproof
\end{lemma}
%-------------------------------------------------

\begin{lemma}\label{aritmetica2}
Let $ \mathbf{D}=\langle D, +, ', 1\rangle $ be a non-trivial orthogroupoid, then the following properties hold: 
\begin{itemize}
\item[1)] $ x+x = x $, for any $ x\in D $; 
\item[2)] $ x\neq x' $ for any $ x\in D $.
\end{itemize}
\proof
1) By axiom (f), $ x+(y+x)=y+x $. Setting $ y=0 $, $ x+x=x+(0+x)= 0+x=x $. \\
2) Suppose by contradiction that $ a=a' $ for some $ a\in D $. Then, by 1), $ a+a=a $ and also $ a'+a= a+a =a $. Then, by (d), $ a=1 $, hence $ 0=1'=1 $. By (b) $ 0+c=c $, for any $ c\in D $, and $ 0'+c = 1+c = 0 +c = c $, thus $ c = 1 $ by (d). So, if $ a=a' $ then $\algd$ is trivial, against the assumption. 
\endproof
\end{lemma}
%--------------------------------------------------
Although in Definition \ref{gruppoide ortogonale} orthogroupoids have a quasi-equational presentation (Condition (d)), we can prove that the same notion can be captured by a single equation, as the following proposition shows:
\begin{proposition}\label{prop: OG is a variety}
A structure $ \mathbf{D}=\langle D, +, ', 1\rangle $ of type $ (2,1,0) $ that satisfies equations (a), (b), (c), (e) and (f) in Definition \ref{gruppoide ortogonale} satisfies condition (d) if and only if it satisfies
\begin{equation}\label{eq:1comm}
  1 +x\approx 1.
\end{equation}

%The class of orthogonal groupoids forms a variety. 
\proof
%It is enough to prove that the quasi-identity (d) can be equivalently substituted with the following identity: $ 1 + x = 1 $. \\
We first derive $ 1+x\approx 1 $, assuming (d). $ 1+(1+x)=1+x $ by axiom (f), and $ 0+ (1+x) = 1+x $ by (b), hence $ 1+x= 1 $ for (d), as desired. \\
For the converse, suppose $ 1 + x \approx 1 $ holds and assume, for $a,b\in D$, that $ a + b=a'+b= b $. First observe that by Lemma \ref{aritmetica1}-(ii), $ (a' + b)'+ a=a $, so $ b' + a=a $. Similarly $ (a+b)' + a' = a' $, hence $ b' + a' = a' $. Now, substituting $ z $ by $ b' $, $ y$ by $a $ and $ x$ by $a'  $ in (e), we obtain $ (((b'+a)'+(b'+a'))'+ (b'+a)') + b = b  $. As $ b' + a = a $ and $ b' + a'= a' $, we get $ b=((a' + a')'+ a')+b = (a''+a')+b = (a+a')+b= 1+b=1 $ as desired. 
\endproof  
\end{proposition}
%---------------------------------------------
\begin{corollary}
The class of orthogroupoids forms a variety axiomatized by equations (a), (b), (c), (e) and (f) in Definition \ref{gruppoide ortogonale} and \eqref{eq:1comm}.
\end{corollary}
%---------------------------------------------
%---------------------------------------------
Let $ \mathbf{D}$ be an orthogroupoid. We now show that the relational system obtained from $\mathbf D$ is an orthogonal relational system.
%---------------------------------------------
%---------------------------------------------
\begin{theorem}\label{sistema indotto}
Let $ \mathbf{D}=\langle D, +, ', 1\rangle $ be an orthogroupoid and $ R $ the induced relation. Then the induced relational system $ \mathbf{A}(D)= \langle D,R,',1\rangle $ is orthogonal and $ R $ is reflexive. 
\proof
By Definition \ref{gruppoide ortogonale}-(a), and Lemma \ref{aritmetica1}-(iv) the mapping $ x\mapsto x' $ is an involution on $ \mathbf{A}(D) $. By Lemma \ref{aritmetica1}-(iii), for all $x$, $ (x,1)\in R $ thus $ \mathbf{A}(D)= \langle D,R,',1\rangle $ is a relational system with 1 and involution. \\
Since  Lemma \ref{aritmetica2}, $x+x=x$, i.e. $R$ is reflexive.\\
To prove that $ \mathbf{A}(D) $ is orthogonal, we verify that conditions (a) and (b) in Definition \ref{sistema ortogonale} are satisfied. \\
By Definition \ref{gruppoide ortogonale}-(c), $ x + x'=1 $ for each $ x\in D $. Obviously $ 1\in U_{R}(x,x') $. Assume $ z\in U_{R}(x,x') $. Then, by definition, $ (x,z)\in R $ and $ (x',z)\in R $ and hence $ x+z=z $ and $ x'+z= z $. Then, axiom (d)  implies $ z=1 $, proving that $ U_{R}(x,x')=\{1\} $. \\
We now prove (b) of Definition \ref{sistema ortogonale}. Assume $ x\neq 0\neq y $ and $ x\perp y $. Then $ (x,y')\in R $ and $ (y,x')\in R $. The following three cases may arise: \\
(i) if $ (x,y)\in R $ then $ (y',x')\in R $ by (iv) of Lemma \ref{aritmetica1}, hence $y+ x'=y'+x'=x'$. Then, by axiom (d), $ x'= 1 $ and $ x=0 $, a contradiction. So this case is impossible.\\
(ii) if $ (x,y)\not\in R $ but $ (y,x)\in R $, then similarly $ y'\in U_{R}(x,x')=\{ 1\} $, whence $ y=0 $, which is again a contradiction. \\
(iii) the last possibility is that $ (x,y)\not\in R $ and $ (y,x)\not\in R $. By axiom (f), $ x+y\in U_{R}(x,y) $. Assume $ z\in U_{R}(x,y) $ with $ z\neq x+y $. Replacing $ x,y,z $ by $ x',y', z' $ in  axiom (e), respectively, we obtain 
\begin{equation}\label{(*)}
  (((z'+y')'+(z'+x'))'+(z'+y'),z)\in R.
\end{equation}
Since $x\perp y$, $ (y,x')\in R $, and so 
\begin{equation}\label{(**)}
y + x' = x'.
\end{equation} 
Moreover, $ z\in U_{R}(x,y) $ yields $ (x,z)\in R $ and $ (y,z)\in R $ thus also $ (z',x')\in R $ and $ (z',y')\in R $, which imply 
\begin{equation}\label{(***)}
z'+x' = x'\text{ and } z'+y'=y'. 
\end{equation}   
Using equations \eqref{(**)} and \eqref{(***)}, we obtain $ x+y= x''+y=(y+x')'+y=((z'+y')'+(z'+x'))'+(z'+y')' $, thus, from equation \ref{(*)}, we conclude $ (x+y,z)\in R $. This proves that $ x+y $ is a supremal element for $ x,y $ and hence $ \mathbf{A}(D) $ is an orthogonal relational system.  
\endproof
\end{theorem}
%----------------------------------------------------
%----------------------------------------------------
Let us remark that a relational system is univocally associated to an orthogroupoid, since the relation $R$ is uniquely determined by the groupoidal operation. \\
A converse of Theorem \ref{sistema indotto} showing how to construct an orthogroupoid out of an orthogonal relational system requires some more lemmas.
%-------------------------------------------------
%------------------------------------------------- 
\begin{lemma}\label{Lemma 3}
Let $ \mathbf{A}=\langle{A,R}\rangle $ be a relational system and let $R$ be reflexive. Then the following equations
\begin{equation}\label{eq:plus}
x+(x+y)\approx x+y\approx y+(x+y) 
\end{equation}
hold in any induced groupoid. 
\proof
Three cases are possible:

\noindent (i) If $ (x,y)\in R $ then $ x+y=y $. Since $R$ is reflexive, also $ (y,y)\in R $, thus $ y\in U_{R}(x,y) $, i.e. $ x+y\in U_{R}(x,y) $ whence $ x+(x+y)=x+y=y+(x+y) $. \\
(ii) If $ (x,y)\not\in R $ but $ (y,x)\in R $ then $ x+y= x $. Using reflexivity of $R$, $ (x,x)\in R $ and hence $ x+y=x\in U_{R}(x,y) $, thus $ x+(x+y)=x+y=y+(x+y) $. \\
(iii) If $ (x,y)\not\in R $ and $ (y,x)\not\in R $ then, by definition, $ x+y$ is arbitrarily chosen in $U_{R}(x,y)$. Hence $ x+(x+y)=x+y=y+(x+y) $. 
\endproof
\end{lemma}
%-------------------------------------------------
%-------------------------------------------------
\begin{lemma}\label{Lemma 4}
Let $ \mathbf{A}=\langle A,R,',1\rangle $ be an orthogonal relational system with $R$ a reflexive relation. If $x,y\in A$, $ x\perp y $ and $ x\neq 0\neq y $, then $ (x,y)\not\in R $ and $ (y,x)\not\in R $.
\proof
Assume $ x\perp y $ and $ x\neq 0\neq y $. Then $ (x,y')\in R $ and $ (y,x')\in R $. Three cases are possible: \\
(i) if $ (x,y)\in R $ then $ (y',x')\in R $ and hence $ x'\in U_{R}(y,y') $. Therefore, $ x'=1 $, i.e. $ x= 0$, a contradiction; \\
(ii) if $ (x,y)\not\in R $ and $ (y,x)\in R $ then $ (x',y')\in R $ and hence $ y'\in U_{R}(x,x') $, whence $ y=0 $, again a contradiction. \\
(iii) The case in which $ (x,y)\in R $ and $ (y,x)\in R $ is ruled out by the previous two. 

Hence the only admissible case is $ (x,y)\not\in R $ and $ (y,x)\not\in R $.
\endproof
\end{lemma}
%-------------------------------------------------
\begin{lemma}\label{lem:0-com}
Let $ \mathbf{A}=\langle A,R,',1\rangle $ be an orthogonal relational system and $ \mathbf{D}=\langle D, +, ', 1\rangle $ be an induced groupoid. Then $\mathbf D$ satisfies 
\begin{equation}\label{eq:x+0=0}
x+0\approx x. 
\end{equation}
\end{lemma}
\begin{proof}
By definition, for any $a\in D$, $(0,a)\in R$. Suppose that $(a,0)\in R$ and $a\not=0$. Then, $(1,a')\in R$. Since $(0,a')\in R$, we get that $\{a',1\} \subseteq U_R(0,1) $, which is a contradiction. Therefore $(a,0)\notin R$, and thus, by Definition \ref{gruppoide indotto}-(ii), $a+0=a$.
\end{proof}
%-------------------------------------------------
\begin{remark}{\rm
Let us notice that in general an orthogroupoid may falsify equation \eqref{eq:x+0=0}, as the orthogroupoid defined by the following table shows ($a+0=b$). }
\begin{center}
\begin{tabular}{|c|cccccc|}
\hline
$\mathbf +$&  $\mathbf 0$ &  $\mathbf 1$ &  $\mathbf a$ &  $\mathbf a'$ &  $\mathbf b$ &  $\mathbf b'$\\
\hline
     $\mathbf 0$ &  $0$ &  $1$ &  $a$ &  $a'$ &  $b$ &  $b'$ \\
     $\mathbf 1$ &  $1$ &  $1$ &  $1$ &  $1$ &  $1$ &  $1$ \\
     $\mathbf a$ &  $b$ &  $1$ &  $a$ &  $1$ &  $b$ &  $1$ \\
     $\mathbf a'$ &  $a'$ &  $1$ &  $1$ &  $a'$ &  $1$ &  $b'$ \\
     $\mathbf b$ &  $a$ &  $1$ &  $a$ &  $1$ &  $b$ &  $1$ \\
     $\mathbf b'$ &  $a'$ &  $1$ &  $1$ &  $a'$ &  $1$ &  $b'$\\
\hline
\end{tabular} 
\end{center}
\end{remark}

We can now prove a converse of Theorem \ref{sistema indotto} for orthogonal relational systems whose relation is both reflexive and transitive. 

%-------------------------------------------------
%-------------------------------------------------
\begin{theorem}\label{Teorema 2}
Let $ \mathbf{A}=\langle A,R,',1\rangle $ be an orthogonal relational system with a reflexive and transitive relation R. Then any groupoid $ \mathbf{G}(A)=\langle A,+,',1\rangle $ induced by $\alga$ is orthogonal.
\end{theorem} 
\proof
Consider an induced groupoid $ \mathbf{G}(A)=\langle A,+,',1\rangle $ as defined in Definition \ref{def: gruppoide indotto da un ORS}.
%Define a binary operation $+$ on the $A$ as in Definition \ref{gruppoide indotto} and, moreover, if $ x\perp y $ and $ x\neq 0 \neq y $ we set $ x+y=y+x=w $, where $ w$ is the supremal element in $ U_{R}(x,y) $. %\textcolor{red}{Finally, define $0 +x=x+0=x$ and $1 +x=x+1=1$.}
% Let us note that this supremal element exists since Definition \ref{sistema ortogonale}, and the operation $+$ is well defined in virtue of Lemma \ref{Lemma 4}. 
We check that $\mathbf{G}(A)$ is an orthogonal groupoid, i.e. it satisfies all the axioms presented in Definition \ref{gruppoide ortogonale}. \\
Axioms (a) and (b) are obviously satisfied. By Lemma \ref{Lemma 3}, $ \mathbf{G}(A) $ satisfies (f). 
Now assume $ x+z=z $ and $ x'+z= z $ for some $ x,z\in A $. Then $ (x,z)\in R $ and $ (x',z)\in R $, thus $ z\in U_{R}(x,x')=\{ 1\} $, i.e. $ z=1 $, proving the quasi-identity (d).
%As regards Equation \eqref{eq:1comm}
%-------------------------------------------------
It remains to show that (c) and (e) hold true. We first prove (e). 
Let $ x,y,z\in A $ and set $ b=(z+y)' $, $ a=((z+y)'+(z+x))' $. By Lemma \ref{Lemma 3} we have $ b+a'=a' $, therefore $ (b,a')\in R $ by Lemma \ref{x+y in U}, whence $ a\perp b $. Let us consider three different cases: \\
Case 1: $ a=0 $, then $ ((z+y)'+(z+x))'+(z+y)')+z'=(z+y)'+z' $. Now, if $ y=0 $ then $ (z+y)'+z'= (z+0)'+z'=z'+z'$, by equation \eqref{eq:x+0=0}, and (e) holds. \\
If $ z=0 $ then $ (z+y)'+z'= y' +1 = 1= z' $, proving (e).\\
If $ z\neq 0\neq y $ then, by reflexivity and Lemma \ref{Lemma 3}, we have $ (z,z+y)\in R $ thus also $ ((z+y)',z')\in R $ and hence $ (z+y)'+z'=z' $, as desired. \\
Case 2: $ b=0 $, then $ ((z+y)'+(z+x))'+(z+y)')+z'=(((0+(z+x))'+0)+z'=(z+x)'+z'=z' $, since by Lemma \ref{x+y in U}, $ (z, z+x)\in R $, and, by definition of the map $()'$, $ ((z+x)', z')\in R $. \\% Lemma \ref{lem:0-com} and the definition of orthogonal system with involution. \\
Case 3: $ a\neq 0\neq b $ and $ a\perp b $. Since Lemma \ref{Lemma 4}, there is a supremal element $ w $ for $ a,b $ in $ U_{R}(a,b) $ and $ w=a+b $. Since $R$ is reflexive, also $ (z,z+y)\in R $ by Lemma \ref{Lemma 3}. However, $ b'=z+y $ thus $ (b,z')\in R $. Since $ a'=(z+y)'+(z+x) $, also $ (z+x,a')\in R $. By Lemma \ref{Lemma 3} $ (z,z+x)\in R $ and, since $R$ is transitive we can conclude $ (z,a')\in R $ and also $ (a,z')\in R $. Altogether we have shown that $ z'\in U_{R}(a,b) $. Since $ a+b $ is a supremal element for $ a,b $, this yields $ (a+b,z')\in R $. Consequently, $ (a+b)+z'=z' $, proving (e). \\
Finally, we show axiom (c). If $ x=0 $ then $ x'=1 $ and hence $ x+x'=0+1=1 $. Similarly for $ x=1 $. If $ x\neq 0 $ and $ x\neq 1 $ then, since $R$ is reflexive, $ x+x'\in U_{R}(x,x')=\{ 1\} $, hence $ x+x'= 1 $  
\endproof
Let us remark that reflexivity and transitivity are necessary conditions to obtain, from Definition \ref{def: gruppoide indotto da un ORS}, an orthogroupoid out of an orthogonal relational system.

\begin{example}\label{example 1}\rm
\textrm{Let $ A=\{0,a,a',1\} $ and
\begin{equation*}
 R=\{ (a,a'), (a',a),(x,1), (0,x)\;\;\forall x\in A\}.
\end{equation*}
 It can be verified that $ \mathbf{A}=\langle A,R,',1\rangle $ is an orthogonal relational system. Indeed: $ U_{R}(0,0')=U_{R}(1,1')=\{1\} $; $ U_{R}(a,a')=\{1\} $ and $ U_{R}(a',a)=\{1\} $. Since $ (a,a')\in R $ we have $ a\perp a $. $ U_{R}(a,a)=U_{R}(a)=\{a',1\} $, thus $ a' $ is a supremal element in $ U_{R}(a,a) $. $U_{R}(a',a')=U_{R}(a')=\{a,1\} $, hence $ a $ is a supremal element in $ U_{R}(a',a') $. This shows $ \mathbf{A}=\langle A,R,',1\rangle $ is an orthogonal relational system: notice that $ R $ is neither reflexive nor transitive.}
 
 \begin{figure}[h]
\begin{tikzpicture}[scale=1.3]
	\draw [line width=1pt, <->](-0.9,0) -- (0.9,0);
	\draw (-1,0) node {$ \bullet $};
	\draw (1,0) node {$ \bullet $};
	\draw (-1.3,0) node {$ a $};
	\draw (1.3,0) node {$ a' $};
	\draw (0,1.5) node {$ \bullet $};
	\draw (0.2,1.75) node {$ 1 $};

	 \draw (0,-1.5) node {$ \bullet $};
	\draw (-0.2, -1.75) node {$ 0 $};
	
	\draw [line width=0.8pt, ->] (0,-1.5) -- (-0.95,-0.1);
	
	\draw [line width=0.8pt, ->] (0,-1.5) -- (0.95,-0.1);
	
	\draw [line width=0.8pt, ->] (-1,0) -- (-0.1,1.4);
	
	\draw [line width=0.8pt, ->] (1,0) -- (0.1,1.4);
	
	\draw [line width=0.9pt, ->] (0,-1.5) -- (0,1.35);

	\end{tikzpicture}

 \caption{The graph of the orthogonal relational system $\mathbf A$.}\label{Grafo: example 1}
\end{figure}
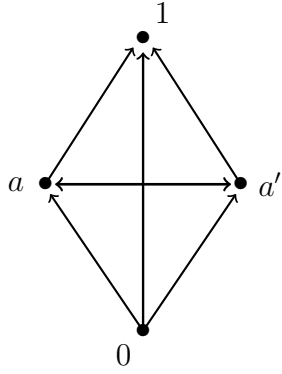

\textrm{An induced groupoid $ \mathbf{G}(A)=\langle A,+,',1 \rangle $ is defined as follows}
\begin{center}
\begin{tabular}{|c|c|c|c|c|}
\hline
+ & $\mathbf{0}$ & $\mathbf{a}$ &$ \mathbf{a'}$ & $\mathbf{1} $\\
\hline
$\mathbf{0}$ & $0$ & $a$ & $a'$ & $1 $\\
\hline
$\mathbf{a}$ & $a$ & $a'$ & $a'$ & $1$ \\
\hline
$\mathbf{a'}$ & $a'$ & $a$ & $a$ & $1$ \\
\hline
$\mathbf{1}$ & $1$ & $1$ & $1$ & $1$ \\
\hline
\end{tabular}
\end{center} 
\textrm{It can be seen that $ \mathbf{G}(A) $ is not an orthogroupoid, since $ a+a'=a'\neq 1 $, against Definition \ref{gruppoide ortogonale}-(c).} 
\end{example}

By Theorem \ref{sistema indotto}, if $\mathbf{G}$ is an orthogroupoid and $R_G$ the induced relation then $R_G$ is reflexive.
In order to prove a converse of this statement, in Theorem \ref{gruppoide indotto} we require, moreover, $R$ to be transitive. In this second example we show that transitivity is a necessary condition to obtain an orthogroupoid out of an orthogonal relational system. 
\begin{example}\label{Example 2}\rm
\textrm{Let $ B=\{0, a, b, a',b',c,c',1 \} $ and a binary relation 
\begin{equation*}
R=\{(a,b),(b,c),(b',a'),(c',b'),(a,c'),(c,a'),(0,x),(x,1),(x,x)\;\;\forall x\in B\}.
\end{equation*}
It can be easily checked that $U_R(a,a')=U_R(b,b')=U_R(c,c')=\{1\}$. The orthogonal pairs are: $a\perp c,c'\perp b, b'\perp a$ and $U_R(a,c)=U_R(c',b)=U_R(b',a)=\{1\}$. Therefore the structure $\mathbf B=\langle B, R,',1\rangle $ is an orthogonal relational system whose relation is reflexive but not transitive. By Definition \ref{def: gruppoide indotto da un ORS} we have that $a+b=b$, $a+c'=c'$ and $c+b=c$ since $(c,b)\not\in R$ but $(b,c)\in R$. Therefore in any groupoid induced by the system $\mathbf B $ axiom (e) in Definition \ref{gruppoide ortogonale} is falsified, indeed: $(((a+c')'+(a+b))'+(a+c')')+a'=((c''+b)'+c'')+a'=(c'+c)+a'=1+a'=1$ since $(1,a')\notin R$ and $(a',1)\in R$, but $a'\not=1$. 
}
 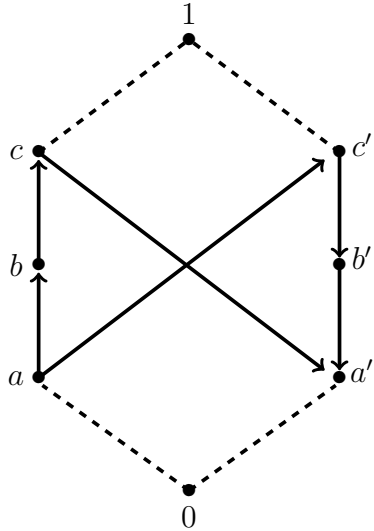
\begin{figure}[h]
\begin{tikzpicture}%[scale=.8]
	
	\draw (-2,0) node {$ \bullet $};
	\draw (2,0) node {$ \bullet $};
	\draw (-2.3,0) node {$ a $};
	\draw (2.3,0.1) node {$ a' $};
	
	\draw (-2,1.5) node {$\bullet $};
	\draw (-2.3, 1.5) node {$ b $};
	
	\draw (2,1.5) node {$\bullet $};
	\draw (2.3,1.6) node {$ b' $};
	
	\draw (0,4.5) node {$ \bullet $};
	\draw (0,4.85) node {$ 1 $};
	
	\draw (0,-1.5) node {$ \bullet $};
	\draw (0, -1.85) node {$ 0 $};
		
	\draw (-2,3) node {$\bullet$}; 
	\draw (-2.3,3) node {$ c $};
	
	\draw (2,3) node {$\bullet$};
	\draw (2.3,3.1) node {$c'$};
		
	\draw [line width=1.4pt, dashed] (0,-1.5) -- (-1.95,-0.1);	 
	
	\draw [line width=1.4pt, dashed] (0,-1.5) -- (1.95,-0.1);
	
	\draw [line width=1.4pt, ->] (-2,0) -- (-2,1.4);
	
	\draw [line width=1.4pt, <-] (2,0.1) -- (2,1.5);
	
	\draw [line width=1.4pt, ->] (-2,1.5) -- (-2,2.9);
	
	\draw [line width=1.4pt, <-] (2,1.6) -- (2,3);
	
	\draw [line width=1.4pt, dashed] (2,3) -- (0,4.5);
	
	\draw [line width=1.4pt, dashed] (-2,3) -- (0,4.5);
	
	\draw [line width=1.4pt, ->] (-2,0) -- (1.8,2.9);
	
	\draw [line width=1.4pt, ->] (-2,3) -- (1.8,0.1);

\end{tikzpicture}

 \caption{The graph representing the orthogonal relational system $\mathbf{B} $ (obvious arrows are omitted).}\label{Grafo: example 2}
\end{figure}

\end{example}

\section{Central elements and decomposition}

The aim of this section is to give a a characterization of the central elements of a variety of orthogroupoids. Contextually a direct decomposition theorem of this variety will follow. The section is based on the ideas developed in \cite{Sal} and \cite{Ledda13} on the general theory of \textit{Church algebras}. 

The notion of Church algebra is based on the simple observation that many well-known algebras, including Heyting algebras, rings with unit and combinatory algebras, possess a term operation $ q $, satisfying the equations: $ q(1,x,y)\approx x $ and $ q(0,x,y)\approx y $. The term operation $ q $ simulates the behaviour of the if-then-else connective and, surprisingly enough, this yields to strong algebraic properties. 

An algebra $\mathbf{A}$\ of type $\nu$\ is a \emph{Church algebra}\ if there
are term definable elements $0^{\mathbf{A}},1^{\mathbf{A}}\in A$\ and a term
operation $q^{\mathbf{A}}$\ s.t., for all $a,b\in A$, $q^{\mathbf{A}}\left(
1^{\mathbf{A}},a,b\right)  =a$\ and $q^{\mathbf{A}}\left(  0^{\mathbf{A}%
},a,b\right)  =b$. A variety $\mathcal{V}$\ of type $\nu$\ is a Church
variety\ if every member of $\mathcal{V}$\ is a Church algebra with respect to
the same term $q\left(  x,y,z\right)  $\ and the same constants $0,1$. 

Taking up an idea from D. Vaggione \cite{Vaggio}, we say that an element \textit{e} of a Church algebra \textbf{A} is \textit{central} if the congruences $ \theta(e,0),\theta (e,1)  $ form a pair of factor congruences on \textbf{A}. A central element $ e $ is nontrivial when $ e\not\in\{0,1\} $. We denote the set of central elements of \textbf{A} (the centre) by $\mathrm {Ce}({A})$.

Setting
$$ x\wedge y=q(x,y,0),\;\; x\vee y= q(x,1,y)\;\; x^{*}=q(x,0,1) $$
we can state the following general result for Church algebras:
\begin{theorem}\label{th: Boolean algebra of centrals}
\emph{\cite{Sal}} Let $ \mathbf{A} $ be a Church algebra. Then 
\[
\mathrm {Ce}(\mathbf{A})=\langle \mathrm {Ce}(A),\wedge,\vee,^{*},0,1\rangle
\]
is a Boolean algebra which is isomorphic to the Boolean algebra of factor congruences of $\mathbf{A}$.
\end{theorem} 

%In what follows, by $ \wedge $, $ \vee $, we mean the operations defined above on Church algebras instead of the usual meet and join operations, respectively. 

If $\mathbf{A}$ is a Church algebra of type $\nu$ and $e\in A$ is a central
element, then we define $\mathbf{A}_{e}=(A_{e},g_{e})_{g\in\nu}$ to be the
$\nu$-algebra defined as follows:

\begin{equation}\label{eq:opAe}
A_{e}=\{e\wedge b:b\in A\};\quad g_{e}(e\wedge\overline{b})=e\wedge
g(e\wedge\overline{b}),
\end{equation}
where $ \overline{b} $ denotes the a n-tuple $ b_1,...,b_n $ and $e\wedge\overline{b} $ is an abbreviation for $ e\wedge b_1,...,e\wedge b_n $.

By \cite[Theorem 4]{Ledda13}, we have that:

\begin{theorem}\label{th: decomposizione Church algebras}
\label{relat}Let $\mathbf{A}$ be a Church algebra of type $\nu$ and $e$ be a
central element. Then we have:

\begin{enumerate}
\item For every $n$-ary $g\in\nu$ and every sequence of elements $\overline
{b}\in A^{n}$, $e\wedge g(\overline{b})=e\wedge g(e\wedge\overline{b})$, so
that the function $h:A\rightarrow A_{e}$, defined by $h(b)=e\wedge b $, is a
homomorphism from $\mathbf{A}$ onto $\mathbf{A}_{e}$.

\item $\mathbf{A}_{e}$ is isomorphic to $\mathbf{A}/\theta(e,1)$. It follows
that $\mathbf{A}=\mathbf{A}_{e}\times\mathbf{A}_{e^{\prime}}$ for every
central element $e$, as in the Boolean case.
\end{enumerate}
\end{theorem}
\noindent
We call \emph{0-commutative} an orthogroupoid if it satisfies
\begin{equation}\label{eq:0-comm-centr}
x + 0\approx 0 +x.
\end{equation}

Let us remark that equation \eqref{eq:0-comm-centr} states a very natural property for orthogroupoids since, in Lemma \ref{lem:0-com}, we proved that any orthogroupoid induced by an orthogonal relational system fulfills this equation.

In the context of 0-commutative orthogroupoids, a new operation $ x\cdot y$ can be defined \emph{\`a la De Morgan} by $(x'+y')'$. 
Few basic properties of $\cdot$ are presented in the following:
\begin{lemma}\label{aritmetica 0-commutativa}
Any 0-commutative orthogroupoid satisfies: 
\begin{itemize}
\item[1)] $ x\cdot 0 \approx 0\cdot x \approx 0 $;
\item[2)] $ x\cdot 1 \approx 1\cdot x\approx x $.
\end{itemize}
\proof

1) $ x\cdot 0 = (x'+0')' = (x' + 1)'= (1+x')' = 1' = 0 $. \\
2) $ x\cdot 1 = (x' + 1')'= (x' + 0)' = (0+x')' = x''= x $.

\endproof
\end{lemma}
\noindent
The following proposition shows that the variety of $0$-commutative orthogroupoids is a Church variety \cite[Definition 3.1]{Sal}.
\begin{proposition}\label{prop:Church variety}
0-commutative orthogroupoids form a Church variety, with witness term 
\[
q(x,y,z) = (x+z)\cdot (x'+y).
\]

\proof
Suppose $ \mathbf{A} $ is a 0-commutative orthogroupoid and $ a,b\in A $. Then, by Lemma \ref{aritmetica 0-commutativa}-(2), $ q(1,a,b) = (1+b)\cdot (0+a) = 1\cdot a = a\cdot 1 = a $. Also, $ q(0,a,b)= (0+b)\cdot (1+a)= b\cdot 1 = b $. 
\endproof
\end{proposition}
According with the results proved in \cite{Sal}, central elements of a Church variety can be described in a very general way.
\begin{proposition}\label{prop: description of central elements in Church varieties}
If $\mb{A}$ is a Church algebra of type $ \nu $ and $ e\in A $, the following conditions are equivalent: 
\begin{itemize}
\item[(1)] e is central;
\item[(2)] for all $ a,b, \vec{a},\vec{b}\in A $: 
\begin{itemize}
\item[\textbf{a)}] $ q(e,a,a)=a $,
\item[\textbf{b)}] $ q(e,q(e,a,b),c)=q(e,a,c)=q(e,a,q(e,b,c)) $,
\item[\textbf{c)}] $ q(e,f(\vec{a}),f(\vec{b}))= f(q(e,a_{1},b_{1}),...,q(e,a_{n},b_{n})) $, for every $ f\in\nu $, 
\item[\textbf{d)}] $ q(e,1,0)=e $.
\end{itemize}

\end{itemize}
\end{proposition}

In case \textbf{A} is a $0$-commutative orthogroupoid, condition (a) reduces to 
\begin{equation}\label{eq: centrali (a)}
(e+a)\cdot (e'+a) = a.
\end{equation}
Conditions (b) read
\begin{equation}\label{eq: centrali (b1)}
(e+c)\cdot (e'+((e+b)\cdot(e'+a))= (e+c)\cdot(e'+a),
\end{equation}
\begin{equation}\label{eq: centrali (b2)} 
(e+c)\cdot(e'+a) = ((e+((e+c)\cdot(e'+b)))\cdot (e'+a).
\end{equation}
\noindent
Condition (c), whenever $ f $ is equal to the constant 1, expresses a property valid for every element. Indeed $ q(e,1,1)= (e + 1)\cdot (e'+1)= 1\cdot 1 = 1 $. If $ f $ coincides with the involution, (c) becomes
\begin{equation}\label{eq: centrali (c1)}
(e+b')\cdot (e'+a')=[(e+b)\cdot (e'+a)]'. 
\end{equation}
\noindent
Finally if $ f $ is equal to $+$, we get: 
\begin{equation}\label{eq: centrali (c2)}
(e+(c+d))\cdot (e'+(a+b))= ((e+c)\cdot (e'+a))+((e+d)\cdot (e'+b)).
\end{equation}

Condition (d) expresses a property that in fact holds for every element: $ e\cdot 1 = e $.
 
\begin{proposition}\label{prop: Boolean algebra of central for OG}
Let $\mb{A}$ be an orthogonal 0-commutative groupoid and $ \mathrm{Ce}({A})$ the set of central elements of $\mathbf{A}$, then $ \mathrm{Ce}(\mathbf{A})=\langle \emph{Ce}(A), +, \cdot, ', 0,1\rangle $ is a Boolean algebra.
\proof
By Theorem \ref{th: Boolean algebra of centrals} we only need to check that $ \vee $, $ \wedge $ and $^{*}$ correspond to $+,\, \cdot,\, '$, respectively. From Lemma \ref{aritmetica 0-commutativa} we obtain:
$$ x\vee y =q(x,1,y)=(x+y)\cdot(x'+1)=(x+y)\cdot 1= x+y $$ 
$$ x^{*} = q(x,0,1) = (x+1)\cdot (x'+0)= 1\cdot (0+x')= 1\cdot x' = x' $$
$$ x\wedge y = (x^{*}\vee y^{*})^{*}= (x'+y')' = x\cdot y $$
\endproof
\end{proposition}

As done in \cite{CGKGLP} for the variety of involutive directoids, we aim at proving a general decomposition result for the variety of $0$-commutative orthogroupoids. Given $ \mathbf{A} $ a 0-commutative orthogroupoid and $ e $ a central element of $ \mathbf{A} $, we define the set 
$$ [0,e]=\{ x\in A:(x,e)\in R, x+e\approx e+x\},$$ 
where $R$ is a relation induced by $ \mathbf{A} $.% and $$ a^{e}= e\cdot a'. $$ 

In the following part of this section we give a decomposition theorem in terms of central elements.
\begin{lemma}\label{lem:Asub=int}
Let $ \mathbf{A} $ be a 0-commutative orthogroupoid and ${e}$ a central element of $ \mathbf{A} $. Then $\mathbf A_e=\langle A_e, +_e,'^{_e},e\rangle$ is the algebra $\mathbf{[0,e]}=\langle [0,e], +, ^{e}, e\rangle $, where for any $a\in [0,e]$ $a^{e}=e\cdot a'$. 
\end{lemma}
\begin{proof}
We first prove that $ A_e =[0,e] $. Suppose $ x\in A_e $, then, by definition of $ A_e $, $ x= e\wedge b $ for some $ b\in A $, i.e. $ x= e\wedge b = q(e,b,0)= e\cdot(e'+b)$. Notice that in any orthogroupoid, $ z' + (z'+(z'+y)')=z'+(z'+y)' $ (condition (f) in Definition \ref{gruppoide ortogonale}), thus by Lemma \ref{aritmetica2} $ (z'+(z'+y)')' + z = z$, i.e. $ (z\cdot (z'+y)) + z = z$. Hence $ x+ e=(e\cdot(e'+b))+e = e$. Furthermore notice that equation \eqref{eq: centrali (b2)}, with $ a=1 $ and $ c=0 $, reads: $ e = (e+(e\cdot (e'+b)). $ Hence we get that $ e+x= e + (e\cdot (e'+b))= e $, proving that $ x\in[0,e] $, hence we have $ A_e\subseteq [0,e] $. \\
For the converse inclusion suppose $ x\in[0,e] $, hence $ (x,e)\in R $ and $ x+e\approx e+x = e $. By the property of central elements expressed by equation \eqref{eq: centrali (a)}, $ x=(e+x)\cdot (e'+x)= e\cdot (e'+x)=q(e,x,0)= e\wedge x. $ 
Thus $ x\in A_{e} $, giving the desired inclusion.

We now prove that, for $x,y\in [0,e]$, $x+_ey=x+y$, where $+_e$ is the operation defined in \eqref{eq:opAe}.
 Let $ x,y\in [0,e] $, then, by definition, $ x+ e=e+x=e $ and $ y+e=e+y=e $. Then, $x+_ey= e\wedge (x+y) = q(e,x+y,0)= q(e,x,0) + q(e,y,0) $ by condition (c) in Proposition \ref{prop: description of central elements in Church varieties}. By definition of $ q $, $ q(e,x,0) + q(e,y,0)= (e\cdot (e'+x))+(e\cdot(e'+y)) $, but since $ e+x=e $ and $ e+y=e $, $ (e\cdot (e'+x))+(e\cdot(e'+y))= ((e+x)\cdot (e'+x))+((e+y)\cdot(e'+y))= x+y $, by equation \eqref{eq: centrali (a)}. Thus $ x+y\in A_e=[0,e] $ as desired. \\ 
As regards $^{e}$ notice that for any $x\in [0,e]$ we have $x^{e}=e\cdot x'=(e+0)\cdot(e'+x')=q(e,x',0)=e\land x'=x'^{_e}$.
\end{proof}

\begin{theorem}\label{th:decomposizione in intervalli}
Let $ \mathbf{A} $ be a $0$-commutative orthogroupoid and ${e}$ a central element of $ \mathbf{A} $. Then $ \mathbf{A}\cong \mathbf{[0,e]}\times\mathbf{ [0,e'] }$.
\proof
Follows directly from Theorem \ref{th: decomposizione Church algebras}, Proposition \ref{prop:Church variety} and Lemma \ref{lem:Asub=int}.
\endproof  
\end{theorem}

Proposition \ref{prop: description of central elements in Church varieties} states that the central elements of a Church variety are characterized by equations. This allows to prove the following
\begin{proposition}\label{prop: atomi di A e Ae}
Let $ \mathbf{A} $ be a $0$-commutative orthogroupoid, $ e\in\mathrm{Ce}(\mathbf{A}) $ and $ c\in A_e $. Then 
$$ c\in\mathrm{Ce}({A}) \Leftrightarrow c\in\mathrm{Ce}({A}_e) $$ 

\proof 
($\Rightarrow $) It follows from the fact that $0$-commutative orthogroupoids forms a Church variety, hence central elements are described by equations. By Theorem \ref{th: decomposizione Church algebras}, $ h:\mathbf{A}\rightarrow\mathbf{A}_e $ is an onto homomorphism such that for every $ a\in A_e $, $ h(a)=a $ and homomorphisms preserve equations. \\
($\Leftarrow $) Since central elements are characterized by equations, if $ c_1 $ is a central element of a $0$-commutative orthogroupoid $ \mathbf{A}_1 $ and $ c_2 $ is a central element of a $ 0 $-commutative orthogroupoid $ \mathbf{A}_2 $, then $ (c_1,c_2)\in\mathrm{Ce}(\mathbf{A}_1 \times\mathbf{A}_2) $, since equations are preserved by direct products. Suppose $ c\in\mathrm{Ce}({A}_{e}) $, the image of $ c $ by the isomorphism of Theorem \ref{th: decomposizione Church algebras} is $ (c,0) $. Since $ 0 $ is always central, we have that $ (c,0) $ is a central element in $ \mathbf{A}_e \times \mathbf{A}_{e'} $, implying that $ c\in\mathrm{Ce}(\mathbf{A}) $, as $ \mathbf{A}\cong\mathbf{A}_e\times\mathbf{A}_{e'} $. 
\endproof 
\end{proposition} 

In Proposition \ref{th: Boolean algebra of centrals} we have proved that $ \mathrm{Ce}(\mathbf{A}) $ is a Boolean algebra. We can consider the set of its atoms and denote them by $At(\mathbf{A}) $.
\begin{lemma}\label{lemma: atomi}
If $ \mathbf{A} $ is an orthogroupoid and $ e $ is an atomic central element of $ \mathbf{A} $, then $ At(\mathbf{A}_{e'})=At(\mathbf{A})\setminus\{e\} $.
\proof
($ \supseteq $) Since $ e $ is an atom of the Boolean algebra $ \mathrm{Ce}(\mathbf{A}) $, for any other atomic central element $ c\in\mathbf{A} $, $ c\cdot e = e\cdot c = 0 $, therefore $ e' + c' =1 $. By equation \eqref{eq: centrali (a)} we get $ (e+c')\cdot (e'+c')=c' $, hence $ (e+c')\cdot 1= e+c'=c' $. Thus $ eRc' $ (for R the relation induced by the orthogroupoid), then $ cRe' $, by Lemma \ref{aritmetica1}. Hence $ c\in\mathbf{A}_{e'} $. By Proposition \ref{prop: atomi di A e Ae}, $ c\in\mathrm{Ce}({A}_{e'}) $. Moreover, if $ d $ is a central element of $ \mathbf{A}_{e'} $ such that $ d <\ c $, then $ d $ is a central element of $ \mathbf{A} $ and since $ c\in At(\mathbf{A}) $ then necessarily $ d=0 $. \\
$ (\subseteq ) $ Suppose $ c\in At(\mathbf{A}_{e'}) $, then in particular $ c $ is a central element of $ \mathbf{A}_{e'} $ and, by Proposition \ref{prop: atomi di A e Ae}, $ c\in\mathrm{Ce}(\mathbf{A}) $. Let $ d\in\mathrm{Ce}(\mathbf{A}) $, with $ c <\ d $, then we have $ d\leq e' $ and therefore $ d\in\mathrm{Ce}(\mathbf{A}_{e'}) $ by Proposition \ref{prop: atomi di A e Ae}. As, by assumption, $ c\in At(\mathbf{A}_{e'})$ then $ d=0 $, which shows that $ c $ is an atomic central. We now claim that $ c\neq e $. Indeed, suppose by contradiction that $ c= e $, then since $ c\leq e' $ we have $ e\leq e' $, i.e. $ e=e\cdot e' = 0 $ which is a contradiction, as $ e $ is atomic central by hypothesis.
\endproof
\end{lemma}

The above lemma allows to prove the following 
\begin{theorem}\label{th:decomposizione in prodotto di d.i.} 
Let $ \mathbf{A} $ be a $0$-commutative orthogroupoid such that $ \mathrm{Ce}(\mathbf{A}) $ is an atomic Boolean algebra with countably many atoms, then
$$ \mathbf{A}=\prod_{e\in At(\mathbf{A})} \mathbf{A}_e $$
is a decomposition of $ \mathbf{A} $ as a product of directly indecomposable algebras. 
\proof
The argument proceeds by induction on the number of elements of $ At(\mathbf{A}) $. If $ 1 $ is the only central atomic element, then $ \mathbf{A} $ is directly indecomposable and clearly $ \mathbf{A}=\mathbf{A}_1 $. If there is an atomic central element $ e\neq 1 $, then $ \mathbf{A}=\mathbf{A}_e\times\mathbf{A}_{e'} $ by Theorem \ref{th: decomposizione Church algebras}. On the other hand $ \mathrm{Ce}(\mathbf{A}_{e})=\{0,e\} $, because if $ \mathbf{A}_e $ had another element, say $ d $, then $ d $ would be a central element of $ \mathbf{A} $ in virtue of Proposition \ref{prop: atomi di A e Ae} and $ 0 <\ d <\ e $ contradicting the fact that $ e $ is an atom. Consequently $ \mathbf{A}_e $ is directly indecomposable. By Lemma \ref{lemma: atomi} $ At(\mathbf{A}_{e'})=At(\mathbf{A})\setminus\{e\} $ and by induction hypothesis, $ \mathbf{A}_{e'}=\prod_{c\in At(\mathbf{A}_{e'})} \mathbf{A}_c $, whence the result readily follows.
\endproof 
\end{theorem}

\section{Amalgamation property}

A \emph{V-formation} (Figure \ref{amalgamo}) is a $5$-tuple $\left(
\mathbf{A},\mathbf{B}_{1},\mathbf{B}_{2},i,j\right)  $ such that
$\mathbf{A,B}_{1},\mathbf{B}_{2}$ are similar algebras, and
$i:\mathbf{A\rightarrow B}_{1},j:\mathbf{A\rightarrow B}_{2}$ are embeddings.
A class $\mathcal{K}$ of similar algebras is said to have the
\emph{amalgamation property} if for every V-formation with
$\mathbf{A},\mathbf{B}_{1},\mathbf{B}_{2}\in\mathcal{K}$ and $A\neq\emptyset$ there
exists an algebra $\mathbf{D}\in\mathcal{K}$ and embeddings $h:\mathbf{B}%
_{1}\mathbf{\rightarrow D},k:\mathbf{B}_{2}\mathbf{\rightarrow D}$ such that
$k\circ j=h\circ i$. In such a case, we also say that $k$ and $h$
\emph{amalgamate} the V-formation $\left(  \mathbf{A},\mathbf{B}%
_{1},\mathbf{B}_{2},i,j\right)  $. $\mathcal{K}$ is said to have the
\emph{strong amalgamation property} if, in addition, such embeddings can be
taken s.t. $h\circ i(\mathbf{A}) =k\circ j\left(  \mathbf{A}\right)  =$ $h\left(  \mathbf{B}%
_{1}\right)  \cap k\left(  \mathbf{B}_{2}\right)  $.

\begin{figure}[h]
\begin{equation}
\vcenter{\xymatrix@R10pt{
                      & & \algb_2 \; \ar @{^(-->} [rrd]  ^ {k} & &       \\
 \alga \; \ar @{^(->} [rru] ^{j} \ar @{_(->} [rrd] _ {i}& &       & & \algd  \\
                      & & \algb_1 \;  \ar @{_(-->} [rru] _{h} & &     
}}
\end{equation}
 \caption{A generic amalgamation schema}\label{amalgamo}
\end{figure}

Amalgamations were first considered for groups by Schreier~\cite{Sch27} in the
form of amalgamated free products. The general form of the $\mathsf{AP}$ was
first formulated by Fra\"{\i}sse~\cite{Fra54}, and the significance of this
property to the study of algebraic systems was further demonstrated in
J\'{o}nsson's pioneering work on the topic~\cite{Jon56, Jon60, Jon61, Jon62}.
The added interest in the $\mathsf{AP}$ for algebras of logic is due to its
relationship with various syntactic interpolation properties. We refer the
reader to~\cite{MMT12} for relevant references and an extensive discussion of
these relationships.

In this section, we show that the variety of orthogroupoids has the strong
amalgamation property.

\begin{theorem}\label{th: amalgamation property}
The variety of orthogroupoids has the strong amalgamation property.
\end{theorem}

\begin{proof}
Let us suppose that we have a V-formation like the solid part of figure
\ref{amalgamo}, and without loss of generality, let us assume that $B_{1}\cap
B_{2}=A$ and $ i(a) = j(a) $, for every $ a\in A $. We are going to give an explicit construction of the amalgam of this
V-formation. Let us consider $ D=B_{1}\cup B_{2} $. We define an operation $ \oplus $ on $ D $ as follows: 
\begin{equation}\label{eq:oplus}
x \oplus y = \left\{
\begin{array}
[c]{l}%
x+^{B_i} y, \;\;\;\;\text{ if }x,y\in B_{i};\\
1, \;\;\;\;\text{ otherwise. }
\end{array}
\right. 
\end{equation}
%Notice that the assumption $B_{1}\cap B_{2}=A$ alone does not guarantee that the operation in \eqref{eq:oplus} is well defined. Indeed, it may happen that, for some $ x,y\in B_{1}\cap B_{2} $, $ x+^{B_1} y \neq x+^{B_2} y $. However, we can overcome this problem by assuming that the maps $ h,k $ are such that, for any $a\in A$, $ (h\circ i) (a) = (k\circ j) (a) $.
%This can be done without any loss of generality, since, for every V-formation $ \langle A, B_1, B_2, i, j\rangle $, the elements can be renamed so to obtain a formation where $ B_1\cap B_2 = A $ and, for any $a\in A$, $ (h\circ i) (a) = (k\circ j) (a) $.
From now on we will drop superscripts whenever no danger of confusion is impending.
%$ + $ instead of $ +_{B_i} $ and it will be clear by the context whether we mean $ +_{B_1} $ or $ +_{B_2} $. 
We can define a complementation $ ^{*} $ in $ D $ as follows: 
\begin{equation}\label{eq:star}
x^{*}=x'^{^{\mb B_i}} 
\end{equation}
Clearly the element 1 belongs to $ D $. We show that $ \mathbf{D}=\langle D, \oplus, ^{*}, 1 \rangle $ is an orthogroupoid. \\
(a) $ 0\oplus x = x $ holds since $ 0_{_D}= 0_{B_1}=0_{B_2} $. \\
(b) $ x\oplus 1 = 1 $, since $ 1_{_D}= 1_{B_1}=1_{B_2} $. \\
(c) notice that $ x\in B_{i} $ with $ i=1,2 $ if and only if $ x'\in B_{i} $, hence $ x\oplus x^{*}= x+x'= 1. $ \\
(d) due to Proposition \ref{prop: OG is a variety} it is enough to show that $ x+1=1+x=1 $. Since $ 1_{D}= 1_{B_1}=1_{B_2} $, $ 1\oplus x= 1+ x =1 $. \\
(e) we have to prove that 
\begin{equation}\label{eq:(e) in D} 
(((x\oplus y)^{*}\oplus(x\oplus z))^{*}\oplus (x\oplus y)^{*})\oplus x^{*}=x^{*}. 
\end{equation} 
We will proceed through a case-splitting argument. \\
%\commento{Non capisco il motivo di assumere che $x,y,z\notin A $.}\hl{In all such cases we assume that} $ x,y,z \not\in B_1\cap B_2 = A $, as it is clear that, for any $x,y,z\in A $, equation \eqref{eq:(e) in D} holds. \\
\textbf{Case 1:} $ x,y,z\in B_{i} $, where $ i\in\{ 1,2\} $. Then equation \eqref{eq:(e) in D} holds since it holds in $ B_{i} $. \\
\textbf{Case 2:} $ x,y\in B_i $, $ z\in B_{j} $, $ z\not \in B_i $, with $ i,j\in \{ 1, 2\} $ and $ i\neq j $. Then $ x\oplus y = x + y $, while $ x\oplus z = 1 $. Then equation \eqref{eq:(e) in D} reads: $ (((x+y)'+1)'+(x+y)')+x'= (0+(x+y)')+x'= (x+y)' + x' = x' $, which holds by Lemma \ref{aritmetica1} (ii). \\
\textbf{Case 3:} $ x\in B_i $, $ y,z\in B_j $, $ x\not\in B_j $, with $ i\neq j $. We then have $ x\oplus y = 1 = x\oplus z $. Therefore $ (1^{*}\oplus (x \oplus z))^{*}\oplus 1^{*})\oplus x^{*}=((0\oplus 1)^{*}\oplus 0)\oplus x^{*}= (0\oplus 0)\oplus x^{*}= 0\oplus x^{*}= x^{*}.  $ \\
\textbf{Case 4:} $ x,z\in B_i $, $ y\in B_j $, $ y\not\in B_i $, with $ i\neq j $. Then $ x\oplus y = 1 $ and $ x\oplus z = x + z $. Equation \eqref{eq:(e) in D} reads: $ ((0 + (x + z))' + 0) + x' = (x + z)' + x' = x' $, by Lemma \ref{aritmetica1} (ii). \\
It can be verified that no other case is possible.\\
%If $ x\in B_{i} $, $ y\in B_{j} $ with $ i\neq j $, then, by \eqref{eq:oplus}, $ x\oplus y = 1 $. We can assume $ z\in B_{j} $ and $ z\not\in B_{i} $. We get $ (1^{*}\oplus (x \oplus z))^{*}\oplus 1^{*})\oplus x^{*}=((0\oplus 1)^{*}\oplus 0)\oplus x^{*}= (0\oplus 0)\oplus x^{*}= 0\oplus x^{*}= x^{*}.  $ In case $ z\in B_{i} $, then  $(((x\oplus y)^{*}\oplus(x\oplus z))^{*}\oplus (x\oplus y)^{*})\oplus x^{*}$ is exactly $((1'^{ }+ (x+  z))'^{ }+  (x+  y)'^{ })+  x'^{ }$ which is equal to $x'^{ }$.
%\\
%The case $ x\in B_{i} $ and $ y,z\in B_j $ with $ i\neq j $ can be easily checked.\\
%If $ x,y\in B_i $ and $ z\in B_{j} $, $ z\not\in B_i $ with $ i\neq j $. Then equation \eqref{eq:(e) in D} reads: $ (((x+y)'+1)'+(x+y)')+x'= (0+(x+y)')+x'= (x+y)' + x' = x' $, which holds by Lemma \ref{aritmetica1} (ii). \\
(f) $ x\oplus (x\oplus y) = x\oplus y $ reduces to $ x+(x+y)=x+y $ if $ x,y\in B_{i} $ and clearly holds. In case $ x\in B_i $ and $ y\in B_j $ and $ x,y\not\in B_i\cap B_j $, with $ i\neq j $, then we get $ x\oplus 1 = 1 $ which always holds. Similarly for $ y\oplus (x\oplus y) = x\oplus y $. \\
It is clear that $ \mathbf{B}_i $ is a subalgebra of $ \mathbf{D} $. Furthermore, by construction, the intersection of $\mathbf{B}_{1}$ and $\mathbf{B}_{2}$ as subalgebras of $\mathbf{D}$ is the algebra $\mathbf{A}$.
Therefore, we have proven that $\mathbf{D}$ is a strong amalgam of $\mathbf{B}_{1}$ and $\mathbf{B}_{2}$.
\end{proof}
As a byproduct of the previous theorem it follows that the orthogonal relational systems induced by the orthogroupoids in a V-formation are amalgamated, as relational structures, in the orthogonal relational system induced by their amalgam.
\medskip

%The construction used in the proof of Theorem \ref{th: amalgamation property} yields the following result also:
%
%\begin{theorem}\label{th:SAPvar}
%Every variety of orthogonal groupoids enjoys the strong amalgamation property. 
%\end{theorem}
%\begin{proof}
%Let $\sigma\approx \tau$ an equation in the language of orthogonal groupoids and let $\mb A_1,\,\mb A_2$ be orthogonal groupoids s.t. $\mb A_1,\mb A_2\models \sigma\approx \tau$. Let $\mathbf D=\langle A_1\cup A_2, \oplus, ^{*}, 1\rangle$ be the orthogonal groupoid where $\oplus$ and $ ^{*}$ are realized as in equations \eqref{eq:oplus} and \eqref{eq:star}, respectively. A routinary but tedious verification shows that, for $i=1$ or $i=2$, $\sigma^{\mb D}= \sigma^{\mb A_i}$ and $\tau^{\mb D}=\tau^{\mb A_i}$. Therefore, $\mb D\models \sigma\approx \tau$. As a consequence if $\left(
%\mathbf{A},\mathbf{B}_{1},\mathbf{B}_{2},i,j\right)$ is a V-formation in a variety $\mc V$ of orthogonal groupoids, then the strong amalgam $\mb D$ belongs to $\mc V$ as well.
%\end{proof}

{\bf Acknowledgement}
The work of the second author is supported by the Project ``New perspectives on residuated posets'' by GA\v{C}R-Grant agency of Czech republic and FWF-Austrian Science Foundation project I 1923-N25. The third author gratefully acknowledges the support of the Italian Ministry of Scientific
Research (MIUR) within the FIRB project ``Structures and Dynamics of Knowledge
and Cognition", Cagliari: F21J12000140001. Finally, we all thank Francesco Paoli and an anonymous referee for their valuable suggestions.


\begin{thebibliography}{10}

\bibitem{Jon56}
J\'{o}nsson B.
\newblock Universal relational structures.
\newblock {\em Math. Scand.}, 4:193--208, 1956.

\bibitem{Jon60}
J\'{o}nsson B.
\newblock Homogeneous universal relational structures.
\newblock {\em Math. Scand.}, 8:137--142, 1960.

\bibitem{Jon61}
J\'{o}nsson B.
\newblock Sublattices of a free lattice.
\newblock {\em Canadian Journal of Mathematics}, 13:146--157, 1961.

\bibitem{Jon62}
J\'{o}nsson B.
\newblock Algebraic extensions of relational systems.
\newblock {\em Math. Scand.}, 11:179--205, 1962.

\bibitem{MMT12}
Metcalfe G., Montagna F., and Tsinakis C.
\newblock Amalgamation and interpolation in ordered algebras.
\newblock {\em Journal of Algebra}, 402:21--82, 2014.

\bibitem{Ch04}
Chajda I.
\newblock Congruences in transitive relational systems.
\newblock {\em Miskolc. Math Notes}, 5:19--23, 2004.

\bibitem{Ch14}
Chajda I.
\newblock An axiomatization of orthogonal posets.
\newblock {\em Soft Computing}, 18:1--4, 2014.

\bibitem{ChLa14}
Chajda I. and L\"{a}nger H.
\newblock Groupoids corresponding to relational systems.
\newblock {\em Miskolc Math. Notes}.
\newblock forthcoming.

\bibitem{ChLa10}
Chajda I. and L\"{a}nger H.
\newblock Quotients and homomorphisms of relational systems.
\newblock {\em Acta Univ. Palack. Olom., Mathematica}, 49:37--47, 2010.

\bibitem{Chbook}
Chajda I. and L\"{a}nger H.
\newblock {\em Directoids. An Algebraic Approach to Ordered Sets}.
\newblock Heldermann Verlag, 2011.

\bibitem{ChLa13}
Chajda I. and L\"{a}nger H.
\newblock Groupoids associated to relational systems.
\newblock {\em Mathematica Bohemica}, 138:15--23, 2013.

\bibitem{CGKGLP}
Chajda I., Gil-F\`erez J., Kola\v{r}\'{\i}k M., Giuntini R., Ledda A., and
  Paoli F.
\newblock On some properties of directoids.
\newblock {\em Soft Computing}, 19:955--964, 2015.

\bibitem{Ledda13}
A.~Ledda, F.~Paoli, and A.~Salibra.
\newblock On semi-{B}oolean-like algebras.
\newblock {\em Acta Univ. Palack. Olom.}, 52:101--120, 2013.

\bibitem{Malcev}
A.~I. Mal'cev.
\newblock On the general theory of algebraic systems.
\newblock {\em Sbornik: {M}athematics}, 35:3--20, 1954.

\bibitem{Sch27}
Schreier O.
\newblock Die untergruppen der freien gruppen.
\newblock {\em Abh. Math. Sem. Univ. Hambur}, 5:161--183, 1927.

\bibitem{Fra54}
Fra\"{\i}sse R.
\newblock Sur l'extension aux relations de quelques propri\'et\'es des ordres.
\newblock {\em Ann. Sci. Ec. Norm. Sup.}, 71:363--388, 1954.

\bibitem{Rig48}
J.~Riguet.
\newblock Relations binaires, fermetures, correspondances de galois.
\newblock {\em Bull. Soc. Math.}, 76:114--155, 1948.

\bibitem{Sal}
A.~Salibra, A.~Ledda, F.~Paoli, and T.~Kowalski.
\newblock Boolean-like algebras.
\newblock {\em Algebra {U}niversalis}, 69(2):113--138, 2013.

\bibitem{Vaggio}
D.~Vaggione.
\newblock Varieties in which the pierce stalks are directly indecomposable.
\newblock {\em Journal of {A}lgebra}, 184:424--434, 1996.

\end{thebibliography}
\end{document}